\documentclass[11pt]{article}
\usepackage{latexsym,amsfonts,amssymb,amsmath,amsthm}
\usepackage{graphicx}

\usepackage[usenames,dvipsnames]{color}
\usepackage{ulem}
\pdfoutput=1

\begin{document}
\setlength{\baselineskip}{16pt}

\parindent 0.5cm
\evensidemargin 0cm \oddsidemargin 0cm \topmargin 0cm \textheight 22cm \textwidth 16cm \footskip 2cm \headsep
0cm

\newtheorem{theorem}{Theorem}[section]
\newtheorem{lemma}{Lemma}[section]
\newtheorem{proposition}{Proposition}[section]
\newtheorem{definition}{Definition}[section]
\newtheorem{example}{Example}[section]
\newtheorem{corollary}{Corollary}[section]

\newtheorem{remark}{Remark}[section]

\numberwithin{equation}{section}

\def\p{\partial}
\def\I{\textit}
\def\R{\mathbb R}
\def\C{\mathbb C}
\def\u{\underline}
\def\l{\lambda}
\def\a{\alpha}
\def\O{\Omega}
\def\e{\epsilon}
\def\ls{\lambda^*}
\def\D{\displaystyle}
\def\wyx{ \frac{w(y,t)}{w(x,t)}}
\def\imp{\Rightarrow}
\def\tE{\tilde E}
\def\tX{\tilde X}
\def\tH{\tilde H}
\def\tu{\tilde u}
\def\d{\mathcal D}
\def\aa{\mathcal A}
\def\DH{\mathcal D(\tH)}
\def\bE{\bar E}
\def\bH{\bar H}
\def\M{\mathcal M}
\renewcommand{\labelenumi}{(\arabic{enumi})}

\def\disp{\displaystyle}
\def\undertex#1{$\underline{\hbox{#1}}$}
\def\card{\mathop{\hbox{card}}}
\def\sgn{\mathop{\hbox{sgn}}}
\def\exp{\mathop{\hbox{exp}}}
\def\OFP{(\Omega,{\cal F},\PP)}
\newcommand\JM{Mierczy\'nski}
\newcommand\RR{\ensuremath{\mathbb{R}}}
\newcommand\CC{\ensuremath{\mathbb{C}}}
\newcommand\QQ{\ensuremath{\mathbb{Q}}}
\newcommand\ZZ{\ensuremath{\mathbb{Z}}}
\newcommand\NN{\ensuremath{\mathbb{N}}}
\newcommand\PP{\ensuremath{\mathbb{P}}}
\newcommand\abs[1]{\ensuremath{\lvert#1\rvert}}

\newcommand\normf[1]{\ensuremath{\lVert#1\rVert_{f}}}
\newcommand\normfRb[1]{\ensuremath{\lVert#1\rVert_{f,R_b}}}
\newcommand\normfRbone[1]{\ensuremath{\lVert#1\rVert_{f, R_{b_1}}}}
\newcommand\normfRbtwo[1]{\ensuremath{\lVert#1\rVert_{f,R_{b_2}}}}
\newcommand\normtwo[1]{\ensuremath{\lVert#1\rVert_{2}}}
\newcommand\norminfty[1]{\ensuremath{\lVert#1\rVert_{\infty}}}
\newcommand{\ds}{\displaystyle}

\title{Coexistence and Extinction   in   Time-Periodic Volterra-Lotka Type Systems with Nonlocal Dispersal}

\author{
Tung Nguyen\\
Department of Mathematical Sciences\\
University of Illinois Springfield\\
Springfield, IL 62703  \\
\\
Nar Rawal\\
Department of Mathematics\\
Hampton University,
 Hampton, VA 23668
 }

\vskip 30mm

\date{}
\maketitle

\noindent {\bf Abstract.} This paper deals with  coexistence and extinction of  time periodic Volterra-Lotka  type competing systems with nonlocal  dispersal. Such issues have already been studied for  time independent systems with nonlocal  dispersal and  time periodic  systems with random dispersal, but have not been studied yet for  time periodic  systems with nonlocal dispersal.  In this paper, the relations between the coefficients representing Malthusian growths, self regulations and competitions of the two species have been obtained  which ensure coexistence and extinction   for the time periodic Volterra-Lotka  type system with nonlocal dispersal.  The underlying environment of the  Volterra-Lotka  type system under consideration has either hostile surroundings, or  non-flux boundary, or is spatially periodic.

\bigskip

\bigskip


\newpage

\section{Introduction}
\setcounter{equation}{0}

Several models have already been derived or have still  being derived to connect mathematics with ecology by many mathematicians and ecologists. Among them Volterra (1860-1940) and Lotka (1880-1949) are the two who contributed  a model  (in 1925) which is well known as  competition model of
two species, i.e.
\begin{equation}
\label{vol-lot-mod}
\begin{cases}
u_t= u(a_1-b_1u-c_1v), \\
v_t= v(a_2-b_2u-c_2v).
\end{cases}
\end{equation}
Since then  it has  drawn special attention of many mathematicians and ecologists  on the following diffusive Volterra-Lotka type two species competition system,
 \begin{equation}
\label{vol-lot-ran-disp}
\begin{cases}
u_t=\nu_1\Delta u + u(a_1(t,x)-b_1(t,x)u-c_1(t,x)v),\quad x\in D  \\
v_t=\nu_2 \Delta v + v(a_2(t,x)-b_2(t,x)u-c_2(t,x)v),\quad  x\in D\\
Bu=Bv=0,\quad x\in \p D,
\end{cases}
\end{equation}
where $\nu_1,\nu_2$ are positive constants, $a_i,b_i,c_i$ $(i=1,2$) are positive smooth functions, $D\subset\RR^N$
is a smooth bounded domain, and $Bu=Bv=0$ are proper boundary conditions. Ecologically, the functions $a_1$, $a_2$ represent the respective growth
rates of the two species, $b_1$, $c_2$ account for self-regulation of the respective species, and $c_1$, $b_2$
account for competition between the two species.
Among those literatures on \eqref{vol-lot-ran-disp}, many were published during $1980s$ (see \cite{CoLa},  \cite{Hi}, \cite{Le}, \cite{Pa}, etc.) and $ 2000 s$ (see \cite{CaCo},  \cite{HeSh},  \cite{HuLoMi}, \cite{HuMiPo},  etc.).

The differential operator  $u\mapsto \Delta u$ in \eqref{vol-lot-ran-disp} describes the random movements of individuals between adjacent locations and is therefore
  also referred to as a {\it random dispersal operator}.
In  reality,
  interactions or movements of individuals of the underlying systems occur between adjacent as well as non-adjacent locations.
Certain integral operators,   which are referred to as {\it nonlocal dispersal operators}, are used to describe nonlocal interactions of individuals in ecology
(see \cite{BaZh}, \cite{ChChRo}, \cite{Cov1}, \cite{Cov}, \cite{Fif}, \cite{GrHiHuMiVi}, \cite{HMMV}, etc.). Recently, a lot
of attention has  been paid to the following Volterra-Lotka type two species competition systems with nonlocal dispersal,
\begin{equation}
\label{vol-lot-non-disper-DC}
\begin{cases}
u_t=\nu_1[\int_{D} \kappa(y-x)u(t,y)dy-u(t,x)]+ u(a_1(t,x)-b_1(t,x)u-c_1(t,x)v),\quad x\in\bar D \cr
v_t=\nu_2[\int_{D} \kappa(y-x)u(t,y)dy-u(t,x)]+ v(a_2(t,x)-b_2(t,x)u-c_2(t,x)v),\quad x\in\bar D, \cr
\end{cases}
\end{equation}
\begin{equation}
\label{vol-lot-non-disper-NC}
\begin{cases}
u_t=\nu_1\int_{D} \kappa(y-x)[u(t,y)-u(t,x)]dy+ u(a_1(t,x)-b_1(t,x)u-c_1(t,x)v),\quad x\in\bar D \cr
v_t=\nu_2\int_{D} \kappa(y-x)[u(t,y)-u(t,x)]dy+ v(a_2(t,x)-b_2(t,x)u-c_2(t,x)v),\quad x\in\bar D, \cr
\end{cases}
\end{equation}
and
\begin{equation}
\label{vol-lot-non-disper-PC}
\begin{cases}
u_t=\nu_1 \int_{\RR^N} \kappa(y-x)[u(t,y)-u(t,x)]dy+ u(a_1(t,x)-b_1(t,x)u-c_1(t,x)v),\quad x\in \RR^N \cr
v_t=\nu_2\int_{\RR^N} \kappa(y-x)[u(t,y)-u(t,x)]dy+ v(a_2(t,x)-b_2(t,x)u-c_2(t,x)v),\quad x\in  \RR^N \cr
u(t,x+p_k{\bf e_k})=u(t,x),\quad v(t,x+p_k{\bf e_k})=v(t,x),\quad k=1,2, \cdots,N,
\end{cases}
\end{equation}
where $\kappa(\cdot)$ is a nonnegative symmetric  smooth function with support $B(0,r)=\{x\in\RR^N\,|\, \|x\|<r\}$ for some $r>0$,
$\int_{\RR^N}\kappa(z)dz=1$, and in \eqref{vol-lot-non-disper-PC}, $a_i(t,x+p_k{\bf e_k})=a_i(t,x)$, $b_i(t,x+p_k{\bf e_k})=b_i(t,x)$,
and $c_i(t,x+p_k{\bf e_k})=c_i(t,x)$ for $i=1,2$, $k=1,2,\cdots,N$, where $p_k>0$ (see \cite{BaLi}, \cite{BaLiSh}, \cite{HeNgSh}, \cite{LiZhZh}, etc.).
 We point out that the works \cite{KaLoSh1} and \cite{KaLoSh2} considered two species competition systems which involves both random and
nonlocal dispersals.

Thanks to the relation between the nonlocal dispersal operator in \eqref{vol-lot-non-disper-DC} (resp. \eqref{vol-lot-non-disper-NC},  \eqref{vol-lot-non-disper-PC}) and the random dispersal operator in
\eqref{vol-lot-ran-disp} with Dirichlet boundary condition (resp.  Neumann boundary condition, periodic boundary condition) (see \cite{CoElRo}, \cite{CoElRoWo}, \cite{ShXi}), \eqref{vol-lot-non-disper-DC} (resp. \eqref{vol-lot-non-disper-NC},  \eqref{vol-lot-non-disper-PC})
is referred to as a Volterra-Lotka type competition system with nonlocal dispersal and Dirichlet type boundary condition
(resp. Neumann type boundary condition, periodic boundary condition).

Coexistence and extinction dynamics is among the central problems investigated for \eqref{vol-lot-ran-disp}, \eqref{vol-lot-non-disper-DC}, \eqref{vol-lot-non-disper-NC},  and \eqref{vol-lot-non-disper-PC}.
Roughly,  we say that \eqref{vol-lot-ran-disp} with time periodic coefficients has a {\it coexistence state} if it has a time periodic solution
$(u^{**}(t,x),v^{**}(t,x))$ with $u^{**}(t,x),v^{**}(t,x)>0$ for $x\in D$. We say that the {\it species $v$ is
eventually driven to extinction} if $\lim_{t\to\infty}v(t,x)=0$ holds for every solution $(u(t,x),v(t,x))$ of
\eqref{vol-lot-ran-disp} with $u(t,x)>0$, $v(t,x)>0$. Note that, by the regularity of solutions for parabolic equations, a coexistence state
$(u^{**}(t,x),v^{**}(t,x))$ of \eqref{vol-lot-ran-disp} (if exists) is continuous in $x$. We say that \eqref{vol-lot-non-disper-DC}, (resp. \eqref{vol-lot-non-disper-NC},  \eqref{vol-lot-non-disper-PC}) with time periodic coefficients has a {\it coexistence state} if it has a time periodic solution
$(u^{**}(t,x),v^{**}(t,x))$ with $u^{**}(t,x),v^{**}(t,x)>0$ for $x\in \bar D$ and being continuous in $x\in\bar D$ (in the case \eqref{vol-lot-non-disper-PC},
$x\in\RR^N$). We say that the {\it species $v$ is
eventually driven to extinction} if $\lim_{t\to\infty}v(t,x)=0$ holds for every solution $(u(t,x),v(t,x))$ of
\eqref{vol-lot-non-disper-DC}, (resp. \eqref{vol-lot-non-disper-NC},  \eqref{vol-lot-non-disper-PC}) with $u(t,x)>0$, $v(t,x)>0$.

 There are many studies on  the coexistence and extinction dynamics of \eqref{vol-lot-ran-disp}
with $a_i$, $b_i$, and $c_i$ being time independent or periodic
 (see \cite{AhLa},  \cite{CoLa}, \cite{FuLG},
\cite{HeSh},  \cite{HsSmWa}, \cite{XQZh},  \cite{ZhPo}, etc.). In \cite{HeNgSh}, the authors studied coexistence and extinction dynamics of  \eqref{vol-lot-non-disper-DC}, \eqref{vol-lot-non-disper-NC},  and \eqref{vol-lot-non-disper-PC} with $a_i$, $b_i$, and $c_i$ being  independent
of $t$.
Consider the following  spectral problems,
\begin{equation} \label{eigenvalue-eq0} \int_D \kappa(y)u(y)dy-u(x)=\lambda u(x),\quad x\in\bar
D,\ u\in C(\bar
D), \end{equation}
\begin{equation} \label{eigenvalue-eq0a} \int_D \kappa(y-x)[u(y)-u(x)]dy=\lambda u(x),\quad x\in\bar
D,\  u\in C(\bar
D), \end{equation}
and
\begin{equation} \label{eigenvalue-eq0b} \int_{\RR^N} \kappa(y-x)[u(y)-u(x)]dy=\lambda u(x),\quad x\in\bar
\RR^N,\ u\in C(
\RR^N). \end{equation}
Let $\lambda_0^D,\quad  \lambda_0^N,\quad \lambda_0^P$ be the principal eigenvalues of \eqref{eigenvalue-eq0}, \eqref{eigenvalue-eq0a}, and \eqref{eigenvalue-eq0b}, respectively. See section 2 for the existence of $\lambda_0^D$,  $\lambda_0^N$ and $\lambda_0^P$. It should be noted that $\lambda_0^D<0$, $\lambda_0^N=0$,  and $\lambda_0^P=0$.

 Set
\begin{equation}
\label{coefficient-eq}
\begin{cases}
a_{iL(M)}=\inf_{t\in\RR,x\in \bar D } (\sup_{t\in\RR,x\in \bar D})a_i(t,x)\cr b_{iL(M)}=\inf_{t\in\RR,x\in \bar D } (\sup_{t\in\RR,x\in \bar
D})b_i(t,x)\cr c_{iL(M)}=\inf_{t\in\RR,x\in \bar D } (\sup_{t\in\RR,x\in \bar D})c_i(t,x)
\end{cases}
\end{equation}
for $i=1,2$, where $D=\RR^N$ in the case of \eqref{vol-lot-non-disper-PC}.
The following two theorems on  the coexistence and extinction  of time independent competing systems with nonlocal dispersal are proved in \cite{HeNgSh}.

\medskip

\noindent {\bf Theorem A'}  (Coexistence states) {\it Consider \eqref{vol-lot-non-disper-DC} with $a_i,b_i,c_i$ ($i=1,2$) being independent of $t$ and
assume that $a_{iL}>-\nu_i \lambda_0^D$ for $i=1,2$.
\begin{itemize}
\item[(1)] If $a_{1L}>-\nu_1 \lambda_0^D+\frac{c_{1M}a_{2M}}{c_{2L}}$ and $a_{2L}>-\nu_2
\lambda_0^D+\frac{b_{2M}a_{1M}}{b_{1L}}$, then \eqref{vol-lot-non-disper-DC} has at least one coexistence state
$(u^{**}(x),v^{**}(x))$.

\item[(2)] If  $\nu_1=\nu_2$, $a_1(x)=a_2(x)$, and $b_1(x)>b_2(x)$, $c_1(x)<c_2(x)$ for $x\in\bar
D$, then \eqref{vol-lot-non-disper-DC} has at least one coexistence state $(u^{**}(x),v^{**}(x))$.

\item[(3)]  If  $\nu_1=\nu_2$, $a_1(x)=a_2(x)$ for $x\in \bar D$, and $b_i$, $c_i$ $(i=1,2)$ are
constant functions with $b_1>b_2$ and $c_1<c_2$, then \eqref{vol-lot-non-disper-DC} has a unique globally stable
coexistence state $(u^{**}(x),v^{**}(x))$.
\end{itemize}
}

\noindent {\bf Theorem B'}  (Extinction) {\it Consider \eqref{vol-lot-non-disper-DC} with $a_i,b_i,c_i$ ($i=1,2$) being independent of $t$ and
assume that $a_{iL}>-\nu_i \lambda_0^D$ for $i=1,2$.
\begin{itemize}
\item[(1)]  If $a_{1L}> \frac{c_{1M}a_{2M}}{c_{2L}}$, $a_{2M}\leq \frac{a_{1L}b_{2L}}{b_{1M}}$,
$\nu_1= \nu_2$, and  $a_{1L}\geq a_{2M}$, then  species $v$ is
eventually driven to extinction.

\item[(2)] If $a_{1M}\leq \frac{c_{1L}a_{2L}}{c_{2M}}$, $a_{2L}> \frac{a_{1M}b_{2M}}{b_{1L}}$,
$\nu_1= \nu_2$, and  $a_{1M}\leq a_{2L}$, then  species $u$ is
eventually driven to extinction.

\item[(3)] If  $\nu_1<\nu_2$, and $a_1(x)=a_2(x)$, $b_1(x)=b_2(x)$,
$c_1(x)=c_2(x)$, then  species $v$ is
eventually driven to extinction.
\end{itemize}
}

Similar results to Theorem A' and Theorem B' have been proved in \cite{HeNgSh} for \eqref{vol-lot-non-disper-NC} and
\eqref{vol-lot-non-disper-PC} (see Theorems C, D, E, and F in \cite{HeNgSh}).

Up to  our knowledge, there is little study on the coexistence and extinction dynamics of \eqref{vol-lot-non-disper-DC},
\eqref{vol-lot-non-disper-NC} and
\eqref{vol-lot-non-disper-PC} with time periodic coefficients. The objective of this paper is to study the coexistence and extinction dynamics of \eqref{vol-lot-non-disper-DC},
\eqref{vol-lot-non-disper-NC} and
\eqref{vol-lot-non-disper-PC} with time periodic coefficients.
 Throughout the rest of this paper, $D=\RR^N$ when \eqref{vol-lot-non-disper-PC}
is considered. We  recall that the following results are proved in \cite{RaSh}.
\begin{itemize}
\item Consider \eqref{vol-lot-non-disper-DC} (resp.
\eqref{vol-lot-non-disper-NC},
\eqref{vol-lot-non-disper-PC}) and assume $a_{iL}>-\nu_i \lambda_0^D$ (resp. $a_{iL}>-\nu_i \lambda_0^N$, $a_{iL}>-\nu_i \lambda_0^P$)
for $i=1,2$. Then \eqref{vol-lot-non-disper-DC} (resp.
\eqref{vol-lot-non-disper-NC},
\eqref{vol-lot-non-disper-PC})
has a  semitrivial  time periodic solution $(u^*(t,\cdot),0)\in (C(\bar
D,\ensuremath{\mathbb{R}})\setminus\{0\})\times C(\bar D,\ensuremath{\mathbb{R}})$ which is globally semi-stable in the sense that for any $u_0\in
C(\bar D,\ensuremath{\mathbb{R}})$ with  $u_0\geq 0$ and $u_0\not\equiv 0$, $(u(t,\cdot;u_0,0),v(t,\cdot;u_0,0))-(u^*(t,\cdot),0)\to (0,0)$ as
$t\to\infty$ (see Proposition \ref{semi-trivial-prop}).

\item Consider \eqref{vol-lot-non-disper-DC} (resp.
\eqref{vol-lot-non-disper-NC},
\eqref{vol-lot-non-disper-PC}) and assume $a_{iL}>-\nu_i \lambda_0^D$ (resp. $a_{iL}>-\nu_i \lambda_0^N$, $a_{iL}>-\nu_i \lambda_0^P$)
for $i=1,2$. Then \eqref{vol-lot-non-disper-DC} (resp.
\eqref{vol-lot-non-disper-NC},
\eqref{vol-lot-non-disper-PC}) has a  semitrivial time periodic solution $(0,v^*(t,\cdot))\in C(\bar
D,\ensuremath{\mathbb{R}})\times(C(\bar D,\ensuremath{\mathbb{R}})\setminus\{0\})$ which is globally semi-stable in the sense that for any $v_0\in
C(\bar D,\ensuremath{\mathbb{R}})$ with $v_0\geq 0$ and $v_0\not\equiv 0$, $(u(t,\cdot;0,v_0)$, $v(t,\cdot;0,v_0))- (0,v^*(t,\cdot))\to (0,0)$ as
$t\to\infty$ (see Proposition \ref{semi-trivial-prop}).
\end{itemize}

We will prove the following theorems in this paper.

\medskip

\noindent {\bf Theorem A.}  (Coexistence states) {\it Consider \eqref{vol-lot-non-disper-DC} $($resp. \eqref{vol-lot-non-disper-NC}, \eqref{vol-lot-non-disper-PC}$)$ and assume that $a_i$, $b_i$ and $c_i$ are  periodic in $t$ with period $T$ and
$a_{iL}>-\nu_i \lambda_0^D$ $($resp. $a_{iL}>-\nu_i \lambda_0^N$, $a_{iL}>-\nu_i \lambda_0^P$$)$ for $i=1,2$.
\begin{itemize}
\item[(1)] If $a_{1L}>-\nu_1 \lambda_0^D+\frac{c_{1M}a_{2M}}{c_{2L}}$ and $a_{2L}>-\nu_2
\lambda_0^D+\frac{b_{2M}a_{1M}}{b_{1L}}$ $($resp. $a_{1L}>-\nu_1 \lambda_0^N+\frac{c_{1M}a_{2M}}{c_{2L}}$ and $a_{2L}>-\nu_2
\lambda_0^N+\frac{b_{2M}a_{1M}}{b_{1L}}$, $a_{1L}>-\nu_1 \lambda_0^P+\frac{c_{1M}a_{2M}}{c_{2L}}$ and $a_{2L}>-\nu_2
\lambda_0^P+\frac{b_{2M}a_{1M}}{b_{1L}}$$)$, then \eqref{vol-lot-non-disper-DC} $($resp. \eqref{vol-lot-non-disper-NC}, \eqref{vol-lot-non-disper-PC}$)$ has at least one coexistence state
$(u^{**}(t,x),v^{**}(t,x))=(u^{**}(t+T,x),v^{**}(t+T,x)))$.

\item[(2)] If  $\nu_1=\nu_2$, $a_1(t,x)=a_2(t,x)$, and $\inf_{x\in\bar D}b_1(t,x)>\sup_{x\in\bar D}b_2(t,x)$, $\sup_{x\in\bar D}c_1(t,x)<\inf_{x\in\bar D}c_2(t,x)$ for $t\in\RR$, then \eqref{vol-lot-non-disper-DC} $($resp.  \eqref{vol-lot-non-disper-NC}, \eqref{vol-lot-non-disper-PC}$)$
    has at least one coexistence state $(u^{**}(t,x),v^{**}(t,x))=(u^{**}(t+T,x),v^{**}(t+T,x)))$.

\item[(3)]  If  $\nu_1=\nu_2$, $a_1(t,x)=a_2(t,x)$ for $x\in \bar D$ and $t\in\RR$, and $b_i$, $c_i$ $(i=1,2)$ are
constant functions with $b_1>b_2$ and $c_1<c_2$, then \eqref{vol-lot-non-disper-DC} $($resp.  \eqref{vol-lot-non-disper-NC}, \eqref{vol-lot-non-disper-PC}$)$ has a unique globally stable
coexistence state $(u^{**}(t,x),v^{**}(t,x))=(u^{**}(t+T,x),v^{**}(t+T,x)))$.
\end{itemize}
}

\noindent {\bf Theorem B.}  (Extinction) {\it Consider \eqref{vol-lot-non-disper-DC} $($resp.  \eqref{vol-lot-non-disper-NC}, \eqref{vol-lot-non-disper-PC}$)$ and assume that $a_i$, $b_i$ and $c_i$ are  periodic in $t$ with period $T$ and
$a_{iL}>-\nu_i \lambda_0^D$ $($resp. $a_{iL}>-\nu_i \lambda_0^N$, $a_{iL}>-\nu_i \lambda_0^P$$)$ for $i=1,2$.
\begin{itemize}
\item[(1)]  If $a_{1L}> \frac{c_{1M}a_{2M}}{c_{2L}}$, $a_{2M}\leq \frac{a_{1L}b_{2L}}{b_{1M}}$,
$\nu_1= \nu_2$, and  $a_{1L}\geq a_{2M}$, then $(u^*(t,x),0)$ is
globally stable and hence species $v$ is
eventually driven to extinction.

\item[(2)] If $a_{1M}\leq \frac{c_{1L}a_{2L}}{c_{2M}}$, $a_{2L}> \frac{a_{1M}b_{2M}}{b_{1L}}$,
$\nu_1= \nu_2$, and  $a_{1M}\leq a_{2L}$, then $(0,v^*(t,x))$ is
globally stable and hence species $u$ is
eventually driven to extinction.

\end{itemize}
}

The above results extend Theorem A' and Theorem B' for time independent Volterra-Lotka type two species competition system with nonlocal dispersal
to time periodic ones. They also extend the existing results on coexistence and extinction dynamics for time periodic Volterra-Lotka type two species competition system with random dispersal to such systems with nonlocal dispersal.

Note that ecologically, Theorem B' (3) indicates that in time independent and spatially inhomogeneous media, the species with slower diffusion
is selected for. Such scenario may not be true for two species competition systems with random dispersal in time periodic and spatially inhomogeneous media (see \cite{HuMiPo}). We conjecture that the scenario may also not be true for two species competition systems with nonlocal dispersal in time periodic and spatially inhomogeneous media.

It should be  pointed out that several difficulties arise in dealing with \eqref{vol-lot-non-disper-DC} $($resp.  \eqref{vol-lot-non-disper-NC}, \eqref{vol-lot-non-disper-PC}$)$ when following the
general approach for \eqref{vol-lot-ran-disp}. This is due to the fact that the solution operator of
\eqref{vol-lot-non-disper-DC} $($resp.  \eqref{vol-lot-non-disper-NC}, \eqref{vol-lot-non-disper-PC}$)$ lacks smoothness and compactness in suitable phase spaces. The main tools employed in the
study of \eqref{vol-lot-non-disper-DC},  \eqref{vol-lot-non-disper-NC}, and \eqref{vol-lot-non-disper-PC} include principal spectral theory for nonlocal dispersal operators with time periodic dependence,  comparison principle for  \eqref{vol-lot-non-disper-DC},  \eqref{vol-lot-non-disper-NC}, and \eqref{vol-lot-non-disper-PC}, and sub- and super-solutions.

The rest of this paper is organized as follows. In section 2, we present some preliminary materials for the use
in later sections.  Sections 3 and 4
are devoted to the proofs of Theorems A and B, respectively.

\section{Preliminary}

In this section, we present some preliminary materials for the use in later sections, including principal spectrum theory for nonlocal dispersal
operators with time periodic dependence,  semitrivial time periodic solutions of time periodic Volterra-Lotka type two species competition systems with
nonlocal dispersal, and comparison principal for Volterra-Lotka type two species competition systems with
nonlocal dispersal.

\subsection{Principal spectrum theory of nonlocal dispersal operators with time periodic dependence}

In this subsection, we present some principal spectrum theory for nonlocal dispersal operators with time periodic dependence.

Let
$$
X_1=X_2=C(\bar D,\RR)
$$
with norm $\|u\|_{X_i}=\sup_{x\in\bar D}|u(x)|$ ($i=1,2$),
$$
X_3=\{u\in C(\RR^N,\RR)\,|\, u(x+p_j{\bf e_j})=u(x)\}
$$
with norm $\|u\|_{X_3}=\sup_{x\in\RR^N}|u(x)|$, and
$$
X_i^+=\{u\in X_i\,|\, u\geq 0\},\quad i=1,2,3,
$$
$$X_i^{++}=\left[\begin{array}{ll}
\{u\in X_i^+\,|\, u(x)>0 \quad \forall \,\, x\in\bar D\},& i=1,2\\
\{u\in X_i^+\,|\, u(x)>0\quad \forall x\in\RR^N\},&i=3.
\end{array}\right.
$$
For given $\nu_i>0$ and  $l_i(\cdot)\in X_i\ (i=1, 2, 3)$, let $L_i^0(\nu_i,l_i): \mathcal{D}(L_i^0(\nu_i,l_i))\subset {X}_i\to {X}_i$
be defined as follows,
$$
(L_1^0(\nu_1,l_1)u)(x)=\nu_1\Big[\int_D
\kappa(y-x)u(y)dy-u(x)\Big]+l_1(x)u(x),
$$
$$
(L_2^0(\nu_2,l_2)u)(x)= \nu_2\Big[\int_D
\kappa(y-x)(u(y)-u(x))dy\Big]+l_2(x)u(x),
$$
and
$$
(L_3^0(\nu_3,l_3)u)(x)=\nu_3\Big[\int_{\RR^N}
\kappa(y-x)u(y)dy-u(x)\Big]+l_3(x)u(x).
$$

Let
$$
\mathcal{X}_1=\mathcal{X}_2=\{u\in C(\RR\times \bar D,\RR)\,|\,
u(t+T,x)=u(t,x)\}
$$
with norm $\|u\|_{\mathcal{X}_i}=\sup_{t\in\RR, x\in\bar D}|u(t,x)|$ ($i=1,2$),
$$
\mathcal{X}_3=\{u\in C(\RR\times\RR^N,\RR)\,|\,
u(t+T,x)=u(t,x+p_i{\bf e_i})=u(t,x)\}
$$
with norm $\|u\|_{\mathcal{X}_3}=\sup_{t\in\RR,x\in\RR^N}|u(t,x)|$, and
$$
\mathcal{X}_i^+=\{u\in\mathcal{X}_i\,|\, u\geq 0\},\quad i=1,2,3.
$$
For given $\nu_i>0$ and $l_i\in\mathcal{X}_i\ (i=1,2,3)$,
 let $L_i(\nu_i,l_i): \mathcal{D}(L_i(\nu_i,l_i))\subset \mathcal{X}_i\to \mathcal{X}_i$
be defined as follows,
$$
(L_1(\nu_1,l_1)u)(t,x)=- u_t(t,x)+\nu_1\Big[\int_D
\kappa(y-x)u(t,y)dy-u(t,x)\Big]+l_1(t,x)u(t,x),
$$
$$
(L_2(\nu_2,l_2)u)(t,x)= - u_t(t,x)+\nu_2\Big[\int_D
\kappa(y-x)(u(t,y)-u(t,x))dy\Big]+l_2(t,x)u(t,x),
$$
and
$$
(L_3(\nu_3,l_3)u)(t,x)=- u_t(t,x)+\nu_3\Big[\int_{\RR^N}
\kappa(y-x)u(t,y)dy-u(t,x)\Big]+l_3(t,x)u(t,x).
$$

\begin{definition}
\label{principal-spectrum-def}
\begin{itemize}
\item[(1)] Let
$$
\lambda_i^0(\nu_i,l_i)=\sup\{{\rm Re}\lambda\,|\, \lambda\in\sigma(L_i^0(\nu_i,l_i))\}
$$
for $i=1,2,3$, where $l_i\in X_i$. $\lambda_i^0(\nu_i,l_i)$ is called the {\rm principal spectrum
point} of $L_i^0(\nu_i,l_i)$ $(i=1,2,3)$. If $\lambda_i^0(\nu_i,l_i)$ is an isolated
eigenvalue of $L_i^0(\nu_i,l_i)$ with a positive eigenfunction $\phi$ (i.e.
$\phi\in {X}_i^+$), then $\lambda_i^0(\nu_i,a_i)$ is called the {\rm
principal eigenvalue} of $L_i^0(\nu_i,l_i)$ or it is said that {\rm
$L_i^0(\nu_i,l_i)$ has a principal eigenvalue} $(i=1,2,3)$.

\item[(2)]
Let
$$
\lambda_i(\nu_i,l_i)=\sup\{{\rm Re}\lambda\,|\, \lambda\in\sigma(L_i(\nu_i,l_i))\}
$$
for $i=1,2,3$, where $l_i\in\mathcal{X}_i$. $\lambda_i(\nu_i,l_i)$ is called the {\rm principal spectrum
point} of $L_i(\nu_i,l_i)$ $(i=1,2,3)$. If $\lambda_i(\nu_i,l_i)$ is an isolated
eigenvalue of $L_i(\nu_i,l_i)$ with a positive eigenfunction $\phi$ (i.e.
$\phi\in \mathcal{X}_i^+$), then $\lambda_i(\nu_i,a_i)$ is called the {\rm
principal eigenvalue} of $L_i(\nu_i,l_i)$ or it is said that {\rm
$L_i(\nu_i,l_i)$ has a principal eigenvalue} $(i=1,2,3)$.
\end{itemize}
\end{definition}

\begin{remark}
For given $1\le i\le 3$ and $l_i(\cdot,\cdot)\in\mathcal{X}_i$, if $l_i(t,x)=l_i(x)$, i.e., $l_i(t,x)$ is independent of
$t$, then $\lambda_i(\nu_i,l_i)=\lambda_i^0(\nu_i,l_i)$.
\end{remark}

For given $1\leq i\leq 3$ and $l_i\in {\mathcal X}_i$, let $\hat l_i(x)$ be the time average of $l_i(t,x)$ ($i=1,2,3$), that is,
\begin{equation}
\label{avg-a-eq}
\hat l_i(x)=\frac{1}{T}\int_0^T l_i(t,x)dt,\quad T>0
\end{equation}
and
\begin{equation}
\label{b-i-eq}
m_i(x)=\begin{cases}-\nu_i\quad {\rm for}\quad i=1,3\cr
-\nu_2\int_D \kappa(y-x)dy\quad {\rm for}\quad i=2.
\end{cases}
\end{equation}
Let
\begin{equation}
\label{domain-eq}
D_i=\begin{cases}
\bar D\quad {\rm for}\quad i=1,2\cr
[0,p_1]\times[0,p_2]\times\cdots\times[0,p_N]\quad {\rm for}\quad i=3.
\end{cases}
\end{equation}

\begin{proposition}
\label{pev-prop1}
Let $\nu_i>0$ and  $l_i\in {\mathcal X}_i$ ($1\leq i\leq 3$) be given. If $\lambda\in\RR$ is an eigenvalue of $L_i(\nu_i,l_i)$ with a positive eigenfunction $\phi(t,x)$, then $\lambda$ is the principal eigenvalue of $L_i(\nu_i,l_i)$. Moreover, $\lambda=\lambda_i(\nu_i,l_i)>\max_{x\in  D_i} (m_i(x)+\hat l_i(x))$. Conversely, if
$\lambda_i(\nu_i,l_i)>\max_{x\in D_i} (m_i(x)+\hat l_i(x))$, then $\lambda_i(\nu_i,l_i)$ is the principal eigenvalue of $L_i(\nu_i,l_i)$. Hence,
 $\lambda_i(\nu_i,l_i)$ is the principal
eigenvalue of $L_i(\nu_i,l_i)$ iff $\lambda_i(\nu_i,l_i)>\max_{x\in D_i} (m_i(x)+\hat l_i(x))$.
\end{proposition}

\begin{proof}
See \cite[Theorem A]{RaSh}.
\end{proof}

\begin{proposition}
\label{pev-prop2}
 Let $\nu_i>0$ and  $l_i\in {\mathcal X}_i$ ($1\leq i\leq 3$) be given.
 The  principal eigenvalue of $L_i(\nu_i,l_i)$ exists
if $m_i(x)+\hat l_i(x)$ is $C^N$, there is some $x_0\in {\rm Int}(D_i)$
in the case $i=1,2$ and $x_0\in D_i$ in the case $i=3$ satisfying
that $m_i(x_0)+\hat l_i(x_0)=\max_{x\in D_i}(m_i(x)+\hat l_i(x))$,
and  the partial derivatives of $m_i(x)+\hat l_i(x)$ up to order
$N-1$ at $x_0$ are zero.
\end{proposition}

\begin{proof}
See \cite[Theorem B]{RaSh}.
\end{proof}

\begin{proposition}
\label{pev-prop3}
\begin{itemize}
\item[(1)] For given $1\le i\le 3$, $\nu_i>0$,  and $l_i,\ \tilde l_i\in\mathcal{X}_i$ with $l_i(t,x)\le \tilde l_i(t,x)$,
$$
\lambda_i(\nu_i,l_i)\le \lambda_i(\nu_i,\tilde l_i).
$$

\item[(2)] For given $1\le i\le 3$, $\nu_i>0$, $l_i\in\mathcal{X}_i$, and any constant $a\in\RR$,
$$
\lambda_i(\nu_i,l_i+a)=\lambda_i(\nu_i,l_i)+a.
$$
\end{itemize}
\end{proposition}

\begin{proof}
(1) It follows from \cite[Propositions 3.2 and 3.10]{RaSh}.

(2) It follows directly from the definition.
\end{proof}

\begin{proposition}
\label{pev-prop4}
For given $1\le i\le 3$, $\nu_i>0$, and $l_i\in \mathcal{X}_i$, if there is
$\phi_i\in \mathcal{X}_i^+\setminus\{0\}$ such that
$$
L_i(\nu_i,l_i)\phi_i=0,
$$
then $\lambda_i(\nu_i,l_i)=0$.
\end{proposition}

\begin{proof}
It follows from \cite[Propositions 3.2 and 3.10]{RaSh}.
\end{proof}

We remark that
$$
\lambda_0^D=\lambda_1(1,0)<0,\quad \lambda_0^N=\lambda_2(1,0)=0,\quad \lambda_0^P=\lambda_3(1,0)=0,
$$
and
$$
\lambda_1^0(\nu_1,0)=\nu_1\lambda_0^D,\quad \lambda_2^0(\nu_2,0)=\nu_2\lambda_0^N,\quad \lambda_3^0(\nu_3,0)=\nu_3\lambda_0^P.
$$

\subsection{Semitrivial time periodic solutions}

In this section, we recall the existence and stability of semitrivial time periodic solutions of  \eqref{vol-lot-non-disper-DC}, \eqref{vol-lot-non-disper-NC}, and \eqref{vol-lot-non-disper-PC}.

First of all, let $X_i$, $X_i^+$, $X_i^{++}$ ($1\leq i\leq 3$) be as in subsection 2.1. Semigroup theory (see \cite{Hen}, \cite{Paz})
guarantees for $(u_0,v_0)\in X_1\times X_1$ (resp. $(u_0,v_0)\in X_2\times X_2$, $(u_0,v_0)\in X_3\times X_3$) that \eqref{vol-lot-non-disper-DC}
 (resp. \eqref{vol-lot-non-disper-NC}, \eqref{vol-lot-non-disper-PC}) has a unique (local) solution
$(u(t,\cdot;u_0,v_0)$, $v(t,\cdot;u_0,v_0))$ with $(u(0,\cdot;u_0,v_0),$ $v(0,\cdot;u_0,v_0))=(u_0,v_0)$. Moreover,
 if $(u_0,v_0)\in X_i\times\{0\}$ ($\{0\}\times X_i$), then
$(u(t,\cdot;u_0,v_0),$ $v(t,\cdot;u_0,v_0))\in X_i\times\{0\}$ ($\{0\}\times X_i$).

\begin{proposition}
\label{semi-trivial-prop} If $a_{iL}>-\nu_i \lambda_0^D$ for $i=1, 2$ (resp. $a_{iL}>-\nu_i \lambda_0^N$ for $i=1, 2$,  $a_{iL}>-\nu_i \lambda_0^P$ for
 $i=1, 2$),
then  \eqref{vol-lot-non-disper-DC} (resp. \eqref{vol-lot-non-disper-NC}, \eqref{vol-lot-non-disper-PC}) has two semitrivial time periodic
solutions $(u^*(t,x),0)$ and $(0,v^*(t,x))$ with $u^*(t,\cdot),\ v^*(t,\cdot)\in X_1^{++}$ (resp. $u^*(t,\cdot),\ v^*(t,\cdot)\in X_2^{++}$,
$u^*(t,\cdot),\ v^*(t,\cdot)\in X_3^{++}$). Moreover, for any $(u_0,v_0)\in (X_1^+\setminus\{0\})\times\{0\}$ (resp. $(u_0,v_0)\in (X_2^+\setminus\{0\})\times\{0\}$, $(u_0,v_0)\in (X_3^+\setminus\{0\})\times\{0\}$,
$$
(u(t,\cdot;u_0,v_0),v(t,\cdot;u_0,v_0))-(u^*(t,\cdot),0)\to (0,0)\quad {\rm as}\quad t\to\infty,
$$
and for any
$(u_0,v_0)\in \{0\}\times (X_1^+\setminus\{0\})$ (resp. $(u_0,v_0)\in \{0\}\times (X_2^+\setminus\{0\})$, $(u_0,v_0)\in
\{0\}\times (X_3^+\setminus\{0\})$,
$$
(u(t,\cdot;u_0,v_0),v(t,\cdot;u_0,v_0))-(0,v^*(t,\cdot))\to (0, 0)\quad {\rm as}\quad t\to\infty,
$$
where $(u(t,\cdot;u_0,v_0),v(t,\cdot;u_0,v_0))$ is the solution of \eqref{vol-lot-non-disper-DC} (resp. \eqref{vol-lot-non-disper-NC}, \eqref{vol-lot-non-disper-PC}) with initial $(u_0,v_0)$.
\end{proposition}

\begin{proof}
We give a proof for \eqref{vol-lot-non-disper-DC}. It can be proved similarly for \eqref{vol-lot-non-disper-NC} and \eqref{vol-lot-non-disper-PC}.

First, we note that
$$
\lambda_1(\nu_1, a_1)\ge \lambda_1(\nu_1,a_{1L})=\nu_1 \lambda_1(1,0)+a_{1L}=\nu_1\lambda_0^D +a_{1L}.
$$
Hence $\lambda_1(\nu_1,a_1)>0$. Then by \cite[Theorem E]{RaSh}, \eqref{vol-lot-non-disper-DC} has a semitrivial periodic solution
$(u^*(t,x),0)$ satisfying that $u^*(t,\cdot)\in X_1^{++}$ and for any $(u_0,0)\in (X_1^+\setminus\{0\})\times \{0\}$,
$$
(u(t,\cdot;u_0,v_0),v(t,\cdot;u_0,v_0))-(u^*(t,\cdot),0)\to (0,0)
$$
as $t\to\infty$, where $(u(t,\cdot;u_0,v_0),v(t,\cdot;u_0,v_0))$ is the solution of \eqref{vol-lot-non-disper-DC} with initial
$(u_0,v_0)$.

Similarly, \eqref{vol-lot-non-disper-DC} has a semitrivial periodic solution
$(0,v^*(t,x))$ satisfying that $v^*(t,\cdot)\in X_1^{++}$ and for any $(u_0,0)\in \{0\}\times (X_1^+\setminus\{0\})$,
$$
(u(t,\cdot;u_0,v_0),v(t,\cdot;u_0,v_0))-(0,v^*(t,\cdot))\to (0,0)
$$
as $t\to\infty$, where $(u(t,\cdot;u_0,v_0),v(t,\cdot;u_0,v_0))$ is the solution of \eqref{vol-lot-non-disper-DC} with initial
$(u_0,v_0)$.
\end{proof}

\subsection{Comparison principle}

In this subsection, we recall a comparison for solutions of  \eqref{vol-lot-non-disper-DC},
 \eqref{vol-lot-non-disper-NC}, and \eqref{vol-lot-non-disper-PC}.

For $u_1,u_2\in X_i$ ($1\leq i\leq 3)$, we define
$$
u_1\leq u_2\,\,\,(u_1\geq u_2)\,\,\,{\rm if}\,\,\, u_2-u_1\in X_i^+\,\,\,(u_1-u_2\in X_i^+),
$$
and
$$
u_1\ll u_2\,\,\,(u_1\gg u_2)\,\,\,{\rm if}\,\,\, u_2-u_1\in X_i^{++}\,\,\,(u_1-u_2\in X_i^{++}).
$$

Define the following orderings in $X_i\times X_i$: \begin{equation} \label{order1} (u_1,v_1)\leq_1
(\ll_1) (u_2,v_2)\quad {\rm if}\quad u_1\leq(\ll) u_2,\, v_1\leq(\ll) v_2, \end{equation}
\begin{equation}
\label{order2} (u_1,v_1)\leq_2(\ll_2)(u_2,v_2)\quad {\rm if}\quad u_1\leq(\ll) u_2,\,v_1\geq(\gg) v_2.
\end{equation}
Observe that $\leq_1$ is the usual order and  $\leq_2$ is  called the {\it competitive order} in the literature.

Let  $\tau >0$  and $(u,v)\in C([0,\tau)\times\bar{D},\ensuremath{\mathbb{R}}^2)$ with $(u(t,\cdot),v(t,\cdot))\in X_1^+\times X_1^+$. Then
$(u,v)$ is called a {\it super-solution} ({\it sub-solution}) of \eqref{vol-lot-non-disper-DC} on $[0,\tau)$ if
$$
\begin{cases}
u_t\geq (\leq)\nu_1 [\int_D k(y-x)u(t,y)dy-u(t,x)]+u[a_1(t,x)-b_1(t,x)u-c_1(t,x)v],\quad x\in\bar D,\cr
v_t\leq (\geq
)\nu_2[\int_D k(y-x)v(t,y)dy-v(t,x)]+v[a_2(t,x)-b_2(t,x)u-c_2(t,x)v],\quad x\in \bar D,
\end{cases}
$$
for $t\in (0,\tau)$. Super-solutions and sub-solutions of \eqref{vol-lot-non-disper-NC} and \eqref{vol-lot-non-disper-PC}
are defined similarly.

\begin{proposition}
\label{comparison-prop}
\begin{itemize}
\item[\rm(1)] Consider  \eqref{vol-lot-non-disper-DC} (resp. \eqref{vol-lot-non-disper-NC}, \eqref{vol-lot-non-disper-PC}).
 For given $(u_0,v_0)\in X_1\times X_1$ (resp. $(u_0,v_0)\in X_2\times X_2$, $(u_0,v_0)\in X_3\times X_3$), if $(0,0)\leq_1(u_0,v_0)$, then $(0,0)\leq_1
(u(t,\cdot;u_0,v_0), v(t,\cdot;u_0,v_0))$ for all $t>0$ at which $(u(t,\cdot;u_0,v_0),v(t,\cdot;u_0,v_0))$
exists, where $(u(t,\cdot;u_0,v_0), v(t,\cdot;u_0,v_0))$ is the solution of  \eqref{vol-lot-non-disper-DC} (resp. \eqref{vol-lot-non-disper-NC}, \eqref{vol-lot-non-disper-PC}) with initial $(u_0,v_0)$.

\item[\rm(2)]  If $(0,0)\leq_1 (u_i(t,\cdot),v_i(t,\cdot))$ for $i=1,2$,
$(u_1(0,\cdot),v_1(0,\cdot))\leq_2 (u_2(0,\cdot)$, $v_2(0,\cdot))$, and
$(u_1(t,x),v_1(t,x))$ is a sub-solution and $(u_2(t,x),v_2(t,x))$ is a super-solution of \eqref{vol-lot-non-disper-DC}
(resp.  \eqref{vol-lot-non-disper-NC}, \eqref{vol-lot-non-disper-PC}) on
$[0,\tau)$, then $(u_1(t,\cdot),v_1(t,\cdot))\leq _2 (u_2(t,\cdot)$, $v_2(t,\cdot))$ for $t\in (0,\tau)$.

\item[\rm(3)]   Consider  \eqref{vol-lot-non-disper-DC} (resp. \eqref{vol-lot-non-disper-NC}, \eqref{vol-lot-non-disper-PC}).
 For given $(u_i,v_i)\in X_1\times X_1$ (resp. $(u_i,v_i)\in X_2\times X_2$, $(u_i,v_i)\in X_3\times X_3$) $(i=1,2)$, if $(0,0)\leq_1 (u_i,v_i)$ for $i=1,2$ and $(u_1,v_1)\leq_2 (u_2,v_2)$, then
$$(u(t,\cdot;u_1,v_1), v(t,\cdot;u_1,v_1))\leq _2
(u(t,\cdot;u_2,v_2),v(t,\cdot;u_2,v_2))$$ for all $t>0$ at which both $(u(t,\cdot;u_1,v_1)$,
$v(t,\cdot;u_1,v_1))$ and $(u(t,\cdot;u_2,v_2)$, $v(t,\cdot;u_2,v_2))$ exist, where $(u(t,\cdot;u_i,v_i), v(t,\cdot;u_i,v_i))$ is the solution of  \eqref{vol-lot-non-disper-DC} (resp. \eqref{vol-lot-non-disper-NC}, \eqref{vol-lot-non-disper-PC}) with initial $(u_i,v_i)$.

\item[\rm(4)] Let $(u_0,v_0)\in X^+_i\times X^+_i$ ($i=1,2,3$), then $(u(t,\cdot;u_0,v_0),v(t,\cdot;u_0,v_0))$
exists for all $t>0$.
\end{itemize}
\end{proposition}

\begin{proof}
It follows from the arguments of \cite[Proposition 3.1]{HeNgSh}.
\end{proof}

\section{Existence, Uniqueness, and Stability  of Coexistence States}

In this section, we investigate the existence, uniqueness, and stability of coexistence states of
 \eqref{vol-lot-non-disper-DC},
\eqref{vol-lot-non-disper-NC}, and
\eqref{vol-lot-non-disper-PC},   and prove Theorem A. We first prove the following theorem.

\begin{theorem}
\label{main-thm}
Assume that
$\frac{\inf_{t\in\RR}b_{1}(t,x)}{\sup_{t\in\RR}b_{2}(t,x)}>\frac{\sup_{t\in\RR}c_{1}(t,x)}{\inf_{t\in\RR}c_{2}(t,x)}$ for each $x\in\bar D$. If $(u^{**}(t,x),v^{**}(t,x))$ is a measurable
coexistence state of \eqref{vol-lot-non-disper-DC} (resp.
\eqref{vol-lot-non-disper-NC},
\eqref{vol-lot-non-disper-PC}),  then $(u^{**}(t,x),v^{**}(t,x))$ is continuous in $x\in\bar D$.
\end{theorem}

Observe that, by Theorem \ref{main-thm}, to prove Theorem A, it suffices to prove the existence of a measurable coexistence state.
To prove Theorem \ref{main-thm}, we
 first prove a lemma.

\begin{lemma}
\label{main-lm1}
Consider
\begin{equation}
\label{main-ode-eq1}
\begin{cases}
u_t=u(a_1(t)-b_1(t)u-c_1(t)v)+d_1(t)\cr
v_t=v(a_2(t)-b_2(t)u-c_2(t)v)+d_2(t),
\end{cases}
\end{equation}
where $b_i(\cdot)$, $c_i(\cdot)$, and $d_i(\cdot)$ ($i=1,2$) are positive continuous periodic functions
with period $T$. Assume 
$$
\frac{b_{1L}}{b_{2M}}>\frac{c_{1M}}{c_{2L}}.
$$
Then \eqref{main-ode-eq1} has a unique time periodic positive solution.
\end{lemma}

\begin{proof}
It follows from Theorem 2.3.1 in \cite{peng-zhao}. In the following, we provide the idea of proof.

First of all, there is a unique time periodic stable solution $u^*(t)$ of
$$
\dot u=u(a_1(t)-b_1(t)u)+d_1(t)
$$
and there is a unique time periodic stable solution $v^*(t)$ of
$$
\dot v=v(a_1(t)-c_2(t)v)+d_2(t)
$$
(see \cite[Proposition 2.2]{peng-zhao}).
Then, by Proposition \ref{comparison-prop},
$$(0,v^*(0))\ll_2(u(T;u^*(0),0),v(T;u^*(0),0))\ll _2(u^*(0),0).
$$
This implies that
$$
 (u((n+1)T;u^*(0), 0),v((n+1)T;u^*(0), 0))\ll_2 (u(nT;u^*(0), 0),v(nT;u^*(0),0))\ll _2(u^*(0),0)
$$
and
$$
(0,v^*(0))\ll_2  (u((n+1)T;u^*(0), 0),v((n+1)T;u^*(0), 0))\ll_2 (u(nT;u^*(0), 0),v(nT;u^*(0),0))
$$

Hence $\lim_{n\to\infty} (u(nT;u^*(0), 0),v(nT;u^*(0), 0))$ exists.
Let
$$
(u_0^+,v_0^+)=\lim_{n\to\infty} (u(nT;u^*(0), 0),v(nT;u^*(0), 0)).
$$  We have that
$$
(u^+(t),v^+(t)):=(u(t;u^+_0,v^+_0),v(t;u^+_0, v^+_0))$$
 is a periodic solution of \eqref{main-ode-eq1}.

 Next, by comparison principle for competition
 systems of ODEs, for any
$(u_0,v_0)\in\RR^+\times\RR^+$ with $(u_0,v_0)\ge _2 (u^+(0),v^+(0))$,
$$
(u^+(t),v^+(t))\le_2 (u(t;u_0,v_0),v(t;u_0,v_0))
$$
for all $t\ge 0$. Note that
$u(t;u_0,v_0)$ satisfies
$$
\dot u=u(a_1(t)-b_1(t)u-c_1(t)v(t;u_0,v_0))+d_1(t)<u(a_1(t)-b_1(t)u)+d_1(t).
$$
Then there is $N^*\ge 1$ such that
$$
u(nT;u_0,v_0)\le u^*(0)
$$
for $n\ge N^*$.  This implies that
$$
(u(nT;u_0,v_0),v(nT;u_0,v_0))\le_2(u^*(0),0)
$$
for $n\ge N^*$.
It can then be proved that
$$
\lim_{t\to\infty}[(u(t;u_0,v_0),v(t;u_0,v_0))-(u^+(t),v^+(t))]=0.
$$

Similarly, we can prove the existence of the limit
$$(u_0^-,v_0^-)=\lim_{n\to\infty}(u(nT;0,v^*(0)),v(nT;0,v^*(0))$$
 and that
$(u^-(t),v^-(t)):=(u(t;u_0^-,v_0^-),v(t;u_0^-,v_0^-))$ is a periodic solution of \eqref{main-ode-eq1} satisfying that
$$
\lim_{t\to\infty}[(u(t;u_0,v_0),v(t;u_0,v_0))-(u^-(t),v^-(t))]=0
$$
for any $(u_0,v_0)\in\RR^+\times\RR^+$ with $(u_0,v_0)\le_2 (u^-(0),v^-(0))$.

It now suffices to prove that
$$
(u^{-}(t),v^{-}(t))\equiv (u^{+}(t),v^{+}(t)).
$$
This can be proved by contradiction.
Assume that
$$
(u^{-}(t),v^{-}(t))\not \equiv (u^{+}(t),v^{+}(t)).
$$
Then
we have
$$
u^{-}(t)<u^{+}(t),\,\,\, v^{-}(t)>v^{+}(t)\quad \forall\,\, t\in \RR.
$$
Observe that
$$
\frac{d}{dt}\ln u^\pm (t)=a_1(t)-b_1(t)u^\pm (t)-c_1(t) v^\pm (t)+\frac{d_1(t)}{u^\pm(t)}
$$
and
$$
\frac{d }{dt}\ln v^\pm (t)=a_2(t)-b_2(t)u^\pm(t)-c_2(t) v^\pm(t)+\frac{d_2(t)}{v^\pm(t)}.
$$
Hence
$$
\frac{d }{dt}\ln \frac{u^-(t)}{u^+(t)}=b_1(t)[u^+(t)-u^-(t)]+c_1(t)[v^+(t)-v^-(t)]+d_1(t)\Big[\frac{1}{u^-(t)}-\frac{1}{u^+(t)}\Big]
$$
and
$$
\frac{d}{dt}\ln \frac{v^-(t)}{v^+(t)}=b_2(t)[u^+(t)-u^-(t)]+c_2(t)[v^+(t)-v^-(t)]+d_2(t)\Big[\frac{1}{v^-(t)}-\frac{1}{v^+(t)}\Big].
$$
It then follows that
$$
0=\int_0^T \frac{d }{dt}\ln \frac{u^-(t)}{u^+(t)}dt> \int_0^T \Big[b_1(t)[u^+(t)-u^-(t)]+c_1(t)[v^+(t)-v^-(t)]\Big]dt
$$
and
$$
0=\int_0^T \frac{d}{dt}\ln \frac{v^-(t)}{v^+(t)}dt <\int_0^T \Big[b_2(t)[u^+(t)-u^-(t)]+c_2(t)[v^+(t)-v^-(t)]\Big]dt.
$$
This implies that
$$
b_{1L}\int_0^T[u^+(t)-u^-(t)]dt<c_{1M}\int_0^T [v^-(t)-v^+(t)]dt
$$
and
$$
b_{2M}\int_0^T[u^+(t)-u^-(t)]dt>c_{1L}\int_0^T [v^{-}(t)-v^{+}(t)]dt.
$$
Hence
$$
\frac{b_{1L}}{c_{1M}}<\frac{\int_0^T [v^{-}(t)-v^{+}(t)]dt}{\int_0^T[u^+(t)-u^{-}(t)]dt}<\frac{b_{2M}}{c_{2L}}.
$$
This is a contradiction.
The theorem is thus proved.
\end{proof}

\begin{proof} [Proof of Theorem \ref{main-thm}]
We prove the theorem for   \eqref{vol-lot-non-disper-DC}. It can be proved similarly for \eqref{vol-lot-non-disper-NC} and \eqref{vol-lot-non-disper-PC}.

For any given $x\in\bar D$,
let $d_1(t,x)=\int_D k(y-x)u^{**}(t,y)dy$ and $d_2(t,x)=\int_D k(y-x)v^{**}(t,y)dy$.
Then $d_1(t,x)$ and $d_2(t,x)$ are positive, periodic in $t$ with period $T$, and smooth in $x$.
For given $x\in\bar D$,
$(u(t;x),v(t;x))=(u^{**}(t,x),v^{**}(t,x))$ satisfies the following competitive systems of ODEs,
\begin{equation}
\label{aux-eq1}
\begin{cases}
u_t(t)=u(t)\big(-\nu_1+a_1(t,x)-b_1(t,x)u(t)-c_1(t,x)v(t)\big)+d_1(t,x)\cr
v_t(t)=v(t)\big(-\nu_2+a_2(t,x)-b_2(t,x)u(t)-c_2(t,x)v(t)\big)+d_2(t,x).
\end{cases}
\end{equation}
By Lemma \ref{main-lm1}, \eqref{aux-eq1} has a unique stable time periodic coexistence state $(\tilde u^{**}(t;x),\tilde v^{**}(t;x))$.
By the smoothness of $a_i(t,x)$, $b_i(t,x)$, $c_i(t,x)$, and $d_i(t,x)$ in $x$ for $i=1,2$, we have that
 $(\tilde u^{**}(t;x),\tilde v^{**}(t;x))$ is continuous in $x$.
Therefore, $(u^{**}(t,x),v^{**}(t,x))= (\tilde u^{**}(t;x),\tilde v^{**}(t;x))$  is continuous in $x$ and the theorem then follows.
\end{proof}

We now prove Theorem A.

\begin{proof} [Proof of Theorem A]  We prove the theorem for \eqref{vol-lot-non-disper-DC} by applying Theorem \ref{main-thm} and  modifying the arguments
in \cite[Theorem A]{HeNgSh}. It can be proved similarly for \eqref{vol-lot-non-disper-NC} and \eqref{vol-lot-non-disper-PC}.

(1) Let $(u^*(\cdot,\cdot),0)$ and $(0,v^*(\cdot,\cdot))$ be the semitrivial  time periodic solutions of \eqref{vol-lot-non-disper-DC}.
Let $K,I:X_1\to X_1$ be given by
$$
(Ku)(x)=\int_D \kappa(y-x)u(y)dy,\quad (I u)(x)=u(x)\quad \forall\,\, u\in X_1.
$$

First, note  that
\begin{align*}
u^*_t(t,x)&=\nu_1[K-I]u^*(t,x)+(a_1(t,x)-b_1(t,x)u^*(t,x))u^*(t,x)
\end{align*}
and
$$ u^*(t,x)\leq
\frac{a_{1M}}{b_{1L}}. $$
 We then have \begin{align*}
a_2(t,x)-b_2(t,x)u^*(t,x)&\ge a_2(t,x)-b_2(t,x)\frac{a_{1M}}{b_{1L}}\\
&\geq a_{2L}-\frac{b_{2M}a_{1M}}{b_{1L}}\\
&>-\nu_2\lambda_0^D. \end{align*} Note that $\lambda_1(\nu_2, a_{2L}-\frac{b_{2M}a_{1M}}{b_{1L}})>0$.
By Proposition \ref{pev-prop2}, $\lambda:=\lambda_1(\nu_2, a_{2L}-\frac{b_{2M}a_{1M}}{b_{1L}})$ is the principal eigenvalue of
$\nu_2[K-I]u+
[a_{2L}-\frac{b_{2M}a_{1M}}{b_{1L}}]I$.
Let $\phi^*(x)$ be a positive principal eigenfunction of $$ \nu_2[K-I]u+
[a_{2L}-\frac{b_{2M}a_{1M}}{b_{1L}}] u=\lambda u. $$ Let $v_\epsilon^+(t,x)=\epsilon \phi^*$ and
$u_\epsilon^+(t,x)\equiv u^*(t,x)$. We have $$
\begin{cases}
(u_\epsilon^+)_t\geq \nu_1[K-I]u_\epsilon^++u_\epsilon^+(a_1(t,x)-b_1(t,x)u_\epsilon^+-c_1(t,x)v_\epsilon^+)\cr
(v_\epsilon^+)_t\leq \nu_2[K-I]v_\epsilon^++v_\epsilon^+(a_2(t,x)-b_2(t,x)u_\epsilon^+-c_2(t,x)v_\epsilon^+)
\end{cases}
$$ for $0<\epsilon\ll 1$. Hence $(u_\epsilon^+(t,x),v_\epsilon^+(t,x))$ is a super-solution of
 \eqref{vol-lot-non-disper-DC}. This implies that
  \begin{align*}(0, v^*(t,\cdot))&\leq_2 (u(t+n_2T,\cdot;u^*(0,\cdot),\epsilon
\phi^*),v(t+n_2T,\cdot;u^*(0,\cdot),\epsilon\phi^*)\\
&\leq_2 (u(t+n_1T,\cdot;u^*(0,\cdot),\epsilon
\phi^*),v(t+n_1T,\cdot;u^*(0,\cdot),\epsilon\phi^*)\\
&\leq_2(u^*(t,\cdot),\epsilon \phi^*) \end{align*} for any $t\ge 0$ and
positive integers $n_2>n_1$. Hence there are Lebesgue measurable functions $u^{**}_{+,\epsilon},v^{**}_{+,\epsilon}:\bar
\RR^+\times D\to [0,\infty)$ such that $$(u(t+nT,x;u^*,\epsilon \phi^*),v(t+nT,x;u^*,\epsilon\phi^*))\to
(u^{**}_{+,\epsilon}(t,x),v^{**}_{+,\epsilon}(t,x))\quad \forall \, t\ge 0,\ x\in\bar D $$ as $n\to\infty$.
Moreover,
$$
u^{**}_{+,\epsilon}(t+T,x)=u^{**}_{+,\epsilon}(t,x),\quad v^{**}_{+,\epsilon}(t+T,x)=v^{**}_{+,\epsilon}(t,x),
$$
and
$$
0\leq u^{**}_{+,\epsilon}(t,x)\leq u^*(t,x),\quad \epsilon \phi^*(x)\leq v^{**}_{+,\epsilon}(t,x)\quad\forall \,\, t\ge 0\,\, x\in\bar
D.
$$
Note that
\begin{equation*}
\begin{split}
&u(t+nT,x;u^*(0,\cdot),\epsilon\phi^*)\\
&=u(nT,x;u^*(0,\cdot),\epsilon\phi^*)\\
&\ +\nu_1\int_0^ {t}\Big[\int_D \kappa(y-x)u(nT+\tau,y;u^*(0,\cdot),\epsilon\phi^*)dy-u(nT+\tau,x;u^*(0,\cdot),\epsilon \phi^*)\Big]d\tau\\
&\quad +\int_0^{t} \Big[u(nT+\tau,x;u^*(0,\cdot),\epsilon \phi^*)(a_1(\tau,x)-b_1(\tau,x)u(nT+\tau,x;u^*(0,\cdot),\epsilon
\phi^*)\\
&\hspace{5cm} -c_1(\tau,x)v(nT+\tau,x;u^*(0,\cdot),\epsilon \phi^*))\Big]d\tau
 \end{split}
\end{equation*}
and
\begin{equation*}
\begin{split}
&v(nT+t,x;u^*(0,\cdot),\epsilon\phi^*)\\
&=v(nT,x;u^*(0,\cdot),\epsilon\phi^*)\\
&\ +\nu_2\int_0^ {t}\Big[\int_D \kappa(y-x)v(nT+\tau,y;u^*(0,\cdot),\epsilon\phi^*)dy-v(nT+\tau,x;u^*(0,\cdot),\epsilon \phi^*)\Big]d\tau\\
&\quad +\int_0^{t} \Big[v(nT+\tau,x;u^*(0,\cdot),\epsilon \phi^*)(a_2(\tau,x)-b_2(\tau,x)u(nT+\tau,x;u^*(0,\cdot),\epsilon
\phi^*)\\
&\hspace{5cm} -c_2(\tau,x)v(nT+\tau,x;u^*(0,\cdot),\epsilon \phi^*))\Big]d\tau
 \end{split}
\end{equation*}
for any $t>0$. $n\in\NN$, and $x\in\bar D$. Letting $n\to\infty$, by Lebesgue dominating convergent theorem,
\begin{align*}
u_{+,\epsilon}^{**}(t,x)=&u^{**}_{+,\epsilon}(0,x)
+\nu_1\int_0^ {t}\Big[\int_D \kappa(y-x)u_{+,\epsilon}^{**}(\tau,y)dy-u_{+,\epsilon}^{**}(\tau,x)\\
&\quad +u_{+,\epsilon}^{**}(\tau,x)(a_1(\tau,x)-b_1(\tau,x)u_{+,\epsilon}^{**}(\tau,x)
 -c_1(\tau,x)v_{+,\epsilon}^{**}(\tau,x)\Big]d\tau
\end{align*}
\begin{align*}
v_{+,\epsilon}^{**}(t,x)=&v^{**}_{+,\epsilon}(0,x)
+\nu_2\int_0^ {t}\Big[\int_D
\kappa(y-x)v_{+,\epsilon}^{**}(\tau,y)dy-v_{+,\epsilon}^{**}(\tau,x)\\
&\quad
+v_{+,\epsilon}^{**}(\tau,x)(a_2(\tau,x)-b_2(\tau,x)u_{+,\epsilon}^{**}(\tau,x)-c_2(\tau,x)v_{+,\epsilon}^{**}(\tau,x)\Big]d\tau
\end{align*}
for all $t>0$ and $x\in\bar D$. It then follows that $(u_{+,\epsilon}^{**}(t,x),v_{+,\epsilon}^{**}(t,x))$  is differentiable in $t$ and satisfies
\eqref{vol-lot-non-disper-DC}.

Similarly, let $\psi^*$ be a positive principal eigenfunction of $$
\nu_1[K-I]u+[a_{1L}-\frac{c_{1M}a_{2M}}{c_{2L}}]u=\lambda u, $$
where $\lambda:=\lambda_1(\nu_1, a_{1L}-\frac{c_{1M}a_{2M}}{c_{2L}})$.
 We have that for $0<\epsilon\ll
1$, there are Lebesgue measurable functions $u^{**}_{-,\epsilon},v^{**}_{-,\epsilon}:\bar\RR^+\times \bar D\to [0,\infty)$ such
that $$
 (u(nT+t,x; \epsilon \psi^*,v^*(0,\cdot)),v(nT+t,x;\epsilon \psi^*,v^*(0,\cdot))\to (u^{**}_{-,\epsilon}(t,x),v^{**}_{-,\epsilon}(t,x))\quad
 \forall \,\, t\ge 0,\, x\in\bar D
$$ as $n\to\infty$,
$$
u_{-,\epsilon}^{**}(t+T,x)=u_{-,\epsilon}^{**}(t,x),\,\,\, v_{-,\epsilon}^{**}(t+T,x)=v_{-,\epsilon}^{**}(t,x),
$$
and
$$ \epsilon \psi^*(x)\leq u^{**}_{-,\epsilon}(t,x),\quad 0\leq
v^{**}_{-,\epsilon}(t,x)\leq v^*(t,x)\quad \forall x\in\bar D. $$
By similar arguments as above, $(u_{-,\epsilon}^{**}(t,x),v_{-,\epsilon}^{**}(t,x))$  is differentiable in $t$ and  satisfies \eqref{vol-lot-non-disper-DC}.

Observe that for $0<\epsilon\ll 1$, $$ \epsilon \psi^*(x)\leq u^{**}_{-,\epsilon}(t,x)\leq
u^{**}_{+,\epsilon}(t,x)\leq u^*(t,x),\,\, \epsilon \phi^*(x)\leq v^{**}_{+,\epsilon}(t,x)\leq
v^{**}_{-,\epsilon}(t,x)\leq v^*(t,x)
$$
for all $t\ge 0$ and $ x\in\bar D$.   From $a_{1L}>-\nu_1
\lambda_0^D+\frac{c_{1M}a_{2M}}{c_{2L}}$ and $a_{2L}>-\nu_2 \lambda_0^D+\frac{b_{2M}a_{1M}}{b_{1L}}$
(note that $\lambda_0^D< 0$), we have
$\frac{b_{1L}}{c_{1M}}>\frac{b_{2M}}{c_{2L}}$. By Theorem \ref{main-thm}, both
$(u^{**}_{-,\epsilon}(t,x),v^{**}_{-,\epsilon}(t,x))$ and $(u^{**}_{+,\epsilon}(t,x),v^{**}_{+,\epsilon}(t,x))$ are in ${\rm Int}X^+\times {\rm Int}X^+$ and hence  are coexistence states of \eqref{vol-lot-non-disper-DC}.

(2)  Let $(u^*(\cdot,\cdot),0)$ and $(0,v^*(\cdot,\cdot))$ be the semitrivial  time periodic solutions of \eqref{vol-lot-non-disper-DC}. Let $\nu=\nu_1(=\nu_2)$ and $a(t,x)=a_1(t,x)(=a_2(t,x))$ for $x\in\bar D$.  Note that \begin{equation} \label{thmb-eq1}
u_t^*(t,x)=\nu[K-I]u^*(t,x)+(a(t,x)-b_1(t,x) u^*(t,x))u^*(t,x). \end{equation}
 By  $\sup_{x\in\bar D} b_2(t,x)<\inf_{x\in \bar D} b_1(t,x)$ for any $t\in\RR$, we have $b_2(t,x)<b_1(t,x)$ for $t\in\RR$ and $x\in\bar D$. Then
$$ a(t,x)-b_2(t,x)u^*(t,x)>a(t,x)-b_1 (t,x)u^*(t,x) \quad \forall\, t\in\RR\,\,  x\in\bar D. $$ Let
$$
\epsilon^*_+=\inf_{t\in\RR, x\in\bar D} (b_1(t,x)-b_2(t,x))u^*(t,x) (>0). $$ Then $$
a(t,x)-b_2(t,x)u^*(t,x)>a(t,x)-b_1(t,x)u^*(t,x)+\frac{\epsilon^*_+}{2}\quad \forall t\in\RR,\,\, x\in\bar D. $$ Hence
$v_\epsilon^+(t,x)=\epsilon u^*(t,x)$ ($0<\epsilon \ll 1$) is a strictly sub-solution of $$
v_t=\nu[K-I]v+(a(t,x)-b_2 (t,x)u^*(t,x))v. $$ By the similar  arguments  as in (1), for $0<\epsilon\ll
1$, there are Lebesgue measurable functions $u^{**}_{+,\epsilon},v^{**}_{+,\epsilon}:\bar\RR^+\times \bar D\to [0,\infty)$ such
that $$ (u(nT+t,x;u^*(0,\cdot),\epsilon u^*(0,\cdot)),v(nT+t,x;u^*(0,\cdot),\epsilon u^*(0,\cdot)))\to
(u^{**}_{+,\epsilon}(t,x),v^{**}_{+,\epsilon}(t,x))\  \forall t\ge 0,\,\,  x\in\bar D $$ as $n\to\infty$,
$$
u^{**}_{+,\epsilon}(t+T,x)=u^{**}_{+,\epsilon}(t,x),\,\,\, v^{**}_{+,\epsilon}(t+T,x)=v^{**}_{+,\epsilon}(t,x),
$$
 and
$(u_{+,\epsilon}^{**}(t,x),v_{+,\epsilon}^{**}(t,x))$ satisfies \eqref{vol-lot-non-disper-DC}.

Similarly, by $\inf_{x\in\bar D}c_2(t,x)>\sup_{x\in\bar D}c_1(t,x)$ for all $t\in\RR$, we have $$ a(t,x)-c_1(t,x)v^*(t,x)>a(t,x)-c_2(t,x) v^*(t,x). $$ Set
$$ \epsilon^*_-=\inf_{t\in\RR,x\in\bar D} (c_2(t,x)-c_1(t,x))v^*(t,x) (>0), $$ then
$$
a(t,x)-c_1(t,x) v^*(t,x)>a(t,x)-c_2(t,x) c^*(t,x)+\frac{\epsilon^*_-}{2}.
$$
Thus, given $0<\epsilon\ll 1$, there are Lebesgue measurable functions
$u^{**}_{-,\epsilon},v^{**}_{-,\epsilon}:\bar\RR^+\times \bar D\to [0,\infty)$ such that $$ (u(t+nT,x;\epsilon
v^*(0,\cdot),v^*(0,\cdot)),v(t+nT,x;\epsilon v^*(0,\cdot),v^*(0,\cdot)))\to (u^{**}_{-,\epsilon}(t,x),v^{**}_{-,\epsilon}(t,x)) \quad\forall t\ge 0,\,\, x\in\bar D
$$ as $n\to\infty$,
$$
u^{**}_{-,\epsilon}(t+T,x)=u^{**}_{-,\epsilon}(t,x),\,\,\, v^{**}_{-,\epsilon}(t+T,x)=v^{**}_{-,\epsilon}(t,x),
$$
 and $(u_{-,\epsilon}^{**}(t,x),v_{-,\epsilon}^{**}(t,x))$ satisfies \eqref{vol-lot-non-disper-DC}.

Then by the similar arguments as in (1),
  $(u^{**}_{\pm,\epsilon}(t,x),v^{**}_{\pm,\epsilon}(t,x))$ belongs to ${\rm Int}X^+\times {\rm Int}X^+$ and hence are
   coexistence states
  of \eqref{vol-lot-non-disper-DC}.

(3) It is a special case of (2). By (2), \eqref{vol-lot-non-disper-DC} has coexistence states. We first prove that the
coexistence state of \eqref{vol-lot-non-disper-DC} is unique.

 Let $(u^{**}(t,x),v^{**}(t,x))$ be any given coexistence state of \eqref{vol-lot-non-disper-DC}. Put $\nu=\nu_1(=\nu_2)$
and $a(\cdot,\cdot)=a_1(\cdot,\cdot)(=a_2(\cdot,\cdot))$. Then
\begin{equation*}
\begin{cases}
u_t^{**}(t,x)=\nu[K-I]u^{**}+u^{**}(a(t,x)-b_1u^{**}-c_2v^{**})+(c_2-c_1)u^{**}v^{**},\quad x\in\bar D\cr
v_t^{**}(t,x)=\nu[K-I]v^{**}+v^{**}(a(t,x)-b_1u^{**}-c_2v^{**})+(b_1-b_2)u^{**}v^{**},\quad x\in\bar D.
\end{cases}
\end{equation*}
Multiplying the first equation by $(b_1-b_2)$ and the second one by $(c_2-c_1)$, we obtain
\begin{align*}
&-(b_1-b_2)u_t^{**}(t,x)+(b_1-b_2)\nu[K-I]u^{**}+(b_1-b_2)u^{**}(a(t,x)-b_1u^{**}-c_2 v^{**})\\
&= -(c_2-c_1)v_t^{**}(t,x)+(c_2-c_1)\nu[K-I]v^{**}+(c_2-c_1)v^{**}(a(t,x)-b_1u^{**}-c_2v^{**}).
\end{align*}
This implies that \begin{equation} \label{uniqueness-eq1}
\phi_t^{**}(t,x)= \nu[K-I]\phi^{**}+ (a(t,x)-b_1u^{**}-c_2
v^{**})\phi^{**}, \end{equation} where $\phi^{**}(t,x)=(b_1-b_2)u^{**}(t,x)-(c_2-c_1)v^{**}(t,x)$. Observe
that \begin{equation} \label{thmb-aux-eq1} a(t,x)-b_1u^{**}(t,x)-c_2 v^{**}(t,x)<a(t,x)-b_1 u^{**}(t,x)-c_1v^{**}(t,x),
\end{equation}
\begin{equation}
\label{thmb-aux-eq2}
u_t^{**}(t,x)=\nu[K-I]u^{**}+(a(t,x)-b_1u^{**}-c_1 v^{**})u^{**}. \end{equation} By
\eqref{thmb-aux-eq2} and Proposition \ref{pev-prop4}, $\lambda(\nu,
a(\cdot,\cdot)-b_1u^{**}-c_1v^{**})$ exists and $\lambda(\nu,a(\cdot,\cdot)-b_1u^{**}-c_1v^{**})=0$. This together with
\eqref{thmb-aux-eq1} implies that \eqref{uniqueness-eq1} has only the trivial solution. Therefore
$\phi^{**}\equiv 0$, that is, \begin{equation} \label{uniqueness-eq2}
v^{**}=\frac{b_1-b_2}{c_2-c_1}u^{**}. \end{equation} By \eqref{uniqueness-eq2}, $u^{**}$ is the
unique positive solution of \begin{equation} \label{uniqueness-eq3}
u_t^{**}=\nu[K-I]u^{**}+\Big[a(t,x)-(b+c_1\cdot \frac{b_1-b_2}{c_2-c_1})u^{**}\Big]u^{**}. \end{equation} By
\eqref{uniqueness-eq2} and \eqref{uniqueness-eq3}, the coexistence state of \eqref{vol-lot-non-disper-DC} is unique.

Next we prove the global stability of the unique coexistence state $(u^{**},v^{**})$. Let $\theta^{*}$ be the
unique time periodic positive solution of \begin{equation} \label{stability-eq1}
u_t= \nu[K-I]u+u(a(t,x)-u)
\end{equation}
(see \cite[Theorem E]{RaSh} for the existence of $\theta^*$).
Then $u^*=\frac{\theta^{*}}{b_1}$ and $v^*=\frac{\theta^*}{c_2}$.

%Note that $b_1>b_2$ and $c_2>c_1$. Then
%$$
%b_1\alpha_++\frac{c_1c_2\beta_+}{b_1}>1\quad \textrm{and}\quad b_2 \alpha_++\frac{c_2c_2\beta_+}{b_2}<1
%$$
%hold for every $\alpha_+,\beta_+>0$ with $\frac{1}{b_1}<\alpha_+<\frac{1}{b_2}$ and $0<\beta_+\ll 1$.
For $\alpha_+,\beta_+>0$ with $\frac{1}{b_1}<\alpha_+<\frac{1}{b_2}$ and $0<\beta_+\ll 1$, let
$u_+=\alpha_+\theta^*$ and $v_+=\beta_+\theta^*$. We then have
$$
\begin{cases}
(u_+)_t\geq \nu[K-I]u_++u_+(a(t,x)-b_1u_+-c_1v_+),\quad x\in\bar D\cr
(v_+)_t\leq \nu[K-I]v_++v_+(a(t,x)-b_2u_+-c_2v_+),\quad x\in\bar
D.
\end{cases}
$$
Therefore, \begin{align*} u(t+n_2T,\cdot;u_+,v_+),v(t+n_2T,\cdot;u_+,v_+))&\leq_2
(u(t+n_1T,\cdot;u_+,v_+),v(t+n_1T,\cdot;u_+,v_+)) \\
&\leq_2 (u_+(t,\cdot),v_+(t,\cdot)) \end{align*} for every
$t\ge 0$ and any positive integers $n_2>n_1$. This implies that $$ u(t,\cdot;u_+,v_+),v(t,\cdot;u_+,v_+))- (u^{**}(t,x),v^{**}(t,x))\to (0,0)\mbox{
as }t\to\infty. $$

Similarly, for $\alpha_-,\beta_->0$ with $\frac{1}{c_2}<\beta_-<\frac{1}{c_1}$
and $0<\alpha_-\ll 1$, let $u_-=\alpha_-\theta^*$ and $v_-=\beta_-\theta^*$. Then
\begin{align*} u(t+n_2T,\cdot;u_-,v_-),v(t+n_2T,\cdot;u_-,v_-))&\geq_2
(u(t+n_1T,\cdot;u_-,v_-),v(t+n_1T,\cdot;u_-,v_-))\\
& \geq_2 (u_-,v_-) \end{align*} for every
$t\ge 0$ and any positive integers $n_2>n_1$, thus $$ u(t,\cdot;u_-,v_-),v(t,\cdot;u_-,v_-))- (u^{**}(t,\cdot),v^{**}(t,\cdot))\to (0,0) \,\, {\rm as}\,\,
t\to\infty. $$

For any given $(u_0,v_0)\in (X^+\setminus\{0\})\times (X^+\setminus\{0\})$ and any $\epsilon>0$, by Proposition
\ref{semi-trivial-prop},
 there is $n\in\NN$
such that $$ (0,v^*+\epsilon)\ll_2 (u(t+nT,\cdot;u_0,v_0),v(t+nT,\cdot;u_0,v_0))\ll_2 (u^*+\epsilon,0)
$$ for $t\geq 0$. Then there are $\alpha_\pm,\beta_\pm>0$ with
$\frac{1}{b_1}<\alpha_+<\frac{1}{b_2}$, $0<\beta_+\ll 1$, and  $\frac{1}{c_2}<\beta_-<\frac{1}{c_1}$,
$0<\alpha_-\ll 1$ such that $$ (\alpha_-\theta^*,\beta_-\theta^*)\leq_2
(u(t,\cdot;u_0,v_0),v(t,\cdot;u_0,v_0))\leq_2 (\alpha_+\theta^*,\beta_+\theta^*) $$ for $t\gg 1$.
It therefore follows that $$ (u(t,\cdot;u_0,v_0),v(t,\cdot;u_0,v_0))-(u^{**}(t,\cdot),v^{**}(t,\cdot))\to (0,0)
$$ as $t\to\infty$. Theorem A is thus proved.
\end{proof}

\section{Extinction}

In this section, we study the extinction dynamics of \eqref{vol-lot-non-disper-DC}, \eqref{vol-lot-non-disper-NC}, and
\eqref{vol-lot-non-disper-PC},
and prove Theorem B.
Let $(u^*(t,x),0)$ and
$(0,v^*(t,x))$ be the two semitrivial periodic solutions of \eqref{vol-lot-non-disper-DC} (resp. \eqref{vol-lot-non-disper-NC},
\eqref{vol-lot-non-disper-PC}). We say that $(u^*,0)$ (resp. $(0,v^*)$) is
{\it globally stable} if for any $(u_0,v_0)\in (X_1^+\setminus\{0\})\times (X_1^+\setminus\{0\})$ (resp. $(u_0,v_0)\in (X_2^+\setminus\{0\})\times (X_2^+\setminus\{0\})$, $(u_0,v_0)\in (X_3^+\setminus\{0\})\times (X_3^+\setminus\{0\})$), $$
(u(t,\cdot;u_0,v_0),v(t,\cdot;u_0,v_0))-(u^*(t,\cdot),0)\to (0,0) $$
$$ ({\rm resp.}\quad (u(t,\cdot;u_0,v_0),v(t,\cdot;u_0,v_0))-(0,v^*(t,\cdot))\to (0,0))
$$ as $t\to\infty$, where $(u(t,\cdot;u_0,v_0),v(t,\cdot;u_0,v_0))$ is the solution of
 \eqref{vol-lot-non-disper-DC} (resp. \eqref{vol-lot-non-disper-NC},
\eqref{vol-lot-non-disper-PC}) with initial $(u_0,v_0)$.

\begin{proof}[Proof of Theorem B] We prove Theorem for \eqref{vol-lot-non-disper-DC}. It can be proved similarly for  \eqref{vol-lot-non-disper-NC}, and
\eqref{vol-lot-non-disper-PC}.

(1) First consider \begin{equation} \label{sub-equation}
\begin{cases}
u_t=\nu[K-I]u+u(a_{1L}-b_{1M}u-c_{1M}v),\quad x\in\bar D\cr v_t=\nu[K-I]v+v(a_{2M}-b_{2L}u-c_{2L}v),\quad x\in\bar D.
\end{cases}
\end{equation} For any given $(u_0,v_0)\in X_1^+\times X_1^+$, let
$(u^-(t,x;u_0,v_0),v^-(t,x;u_0,v_0))$ be the solution of \eqref{sub-equation} with
$(u^-(0,x;u_0,v_0),v^-(0,x;u_0,v_0))=(u_0(x),v_0(x))$.
 Let $(u_-^*,0)$ and $(0,v_-^*)$ be the semitrivial equilibria of \eqref{sub-equation}.

By the arguments of \cite[Theorem B]{HeNgSh},  for any $(u_0,v_0)\in (X_1^+\setminus\{0\})\times (X_1^+\setminus\{0\})$, $$
\lim_{t\to\infty} (u^-(t,\cdot;u_0,v_0),v^-(t,\cdot;u_0,v_0))=(u_-^*,0). $$

For any given $(u_0,v_0)\in (X_1^+\setminus\{0\})\times (X_1^+\setminus\{0\})$, for any $\epsilon_2>0$, by
Proposition \ref{semi-trivial-prop}, there is $\tilde T>0$ such that $$ (0,(1+\epsilon_2)v^*_-)\ll_2
(u^-(t,\cdot;u_0,v_0),v^-(t,\cdot;u_0,v_0)) $$ for $t\geq \tilde T$. Then there is $\epsilon_1>0$ such
that $$ (\epsilon_1 v^*_-,(1+\epsilon_2)v^*_-)\leq_2 (u^-(t,\cdot;u_0,v_0),v^-(t,\cdot;u_0,v_0))
$$ for $t\gg 1$. Thus we have $$ (u^-(t,\cdot;u_0,v_0),v^-(t,\cdot;u_0,v_0))\to
(u^*_-,0) $$ as $t\to\infty$, and the claim is proved.

For any given $(u_0,v_0)\in X_1^+\times X_1^+$, by Proposition \ref{comparison-prop}, $$
(u^-(t,\cdot;u_0,v_0),v^-(t,\cdot;u_0,v_0))\leq_2 (u(t,\cdot;u_0,v_0),v(t,\cdot;u_0,v_0))\quad \forall t>0.
$$ By the above claim, for any $(u_0,v_0)\in (X_1^+\setminus\{0\})\times (X_1^+\setminus\{0\})$,
$$ \lim_{t\to\infty} (u^-(t,\cdot;u_0,v_0),v^-(t,\cdot;u_0,v_0))=(u_-^*,0). $$
This together with Proposition \ref{semi-trivial-prop} implies that $$ \lim_{t\to\infty} \big[(u(t,\cdot;u_0,v_0),v(t,\cdot;u_0,v_0))-(u^*(t,\cdot),0)\big]=(0,0).
$$

(2) can be proved by the similar arguments as in (1).
\end{proof}

\end{document}